\documentclass[a4paper,11pt]{amsart} 
\usepackage{amssymb, amsmath}
\usepackage{amscd}
\numberwithin{equation}{section}
\usepackage{epsfig}
\usepackage{amsmath}
\usepackage{amsfonts,amssymb,amsopn}
\usepackage{caption}

\usepackage{amsthm}
\usepackage{hyperref}
\usepackage{verbatim}
\usepackage{color}
\usepackage[color,all]{xy}
\usepackage[utf8]{inputenc}
\usepackage{soul}

\newcommand{\cD}{{\mathcal D}}
\newcommand{\cE}{{\mathcal E}}
\newcommand{\cF}{{\mathcal F}}
\newcommand{\cO}{{\mathcal O}}
\newcommand{\cQ}{{\mathcal Q}}
\newcommand{\cS}{{\mathcal S}}

\newcommand{\PP}{{\mathbb P}}

\newcommand{\Ker}{\mathrm{Ker}}
\newcommand{\Image}{\mathrm{Im}\,}

\newcommand{\Hom}{\mathrm{Hom}}
\newcommand{\rank}{\mathrm{rk}\,}
\newcommand{\Pic}{\mathrm{Pic}}

\newcommand{\Kc}{K_{C}}
\newcommand{\Oc}{{{\cO}_C}}

\newcommand{\Quot}{\mathrm{Quot}}

\newcommand{\mev}{{M_{E , V}}}
\newcommand{\mevo}{{M_{E_1 , V_1}}}
\newcommand{\mevd}{{M^\vee_{E, V}}}
\newcommand{\mevod}{{M^\vee_{E_1, V_1}}}

\newcommand{\Ce}{C_e}

\newcommand{\ope}{\cO_{\PP E^\vee}}
\newcommand{\opeo}{\ope (1)}

\newcommand{\Vef}{V^{e-f}_e}
\newcommand{\Veo}{V^{e-1}_e}

\newcommand{\Qef}{Q^{e-f}_e}

\newcommand{\Qeo}{Q^{e-1}_e}

\newcommand{\alps}{{\alpha_\mathrm{S}}}
\newcommand{\alpl}{{\alpha_\mathrm{L}}}

\newtheorem{theorem}{{\textbf Theorem}}[section]
\newtheorem{proposition}[theorem]{{\textbf Proposition}}

\newtheorem{lemma}[theorem]{{\textbf Lemma}}
\newtheorem{exa}[theorem]{{\textbf Example}}
\newtheorem{conjecture}{{\textbf Conjecture}}
\newtheorem{remit}[theorem]{{\textbf Remark}}
\newtheorem{defn}[theorem]{{\textbf Definition}}
\newenvironment{remark}{\begin{remit}\rm}{\end{remit}}
\newenvironment{definition}{\begin{defn}\rm}{\end{defn}}
\newenvironment{example}{\begin{exa}\rm}{\end{exa}}
\newcounter{tmp}

\title{On $\alpha$-stability and linear stability of generated coherent systems}

\author{Abel Castorena}
\address{Centro de Ciencias Matem\'aticas -- UNAM Campus Morelia, Antigua Carretera a P\'atzcuaro \# 8701, Col.\ Ex Hacienda San Jos\'e de la Huerta, Morelia, Michoac\'an, Mexico C.\ P.\ 58089.}
\email{abel@matmor.unam.mx}

\author{George H.\ Hitching}
\address{Oslo Metropolitan University, Postboks 4, St. Olavs plass, 0130 Oslo, Norway.}
\email{gehahi@oslomet.no}



\begin{document}

\dedicatory{In memory of Professor Peter E.\ Newstead}

\begin{abstract}
There is a well studied notion of GIT-stability for coherent systems over curves, which depends on a real parameter $\alpha$. For generated coherent systems, there is a further notion of stability derived from Mumford's definition of linear stability for varieties in projective space. Let $\alps$ be close to zero and $\alpl \gg 0$. We show that a generated coherent system which is $\alps$-stable and linearly stable is $\alpl$-stable, and give examples showing that without further assumptions, there are no other implications between these three types of stability. We observe that several of the systems constructed have stable dual span bundle, including systems which are not $\alpha$-semistable for any value of $\alpha$. We use this to prove a case of Butler's conjecture for systems of type $(2, d, 5)$.
\end{abstract}

\maketitle

\section{Introduction}

Let $C$ be a complex projective smooth curve. A \textsl{coherent system of type $(r, d, n)$ over $C$} is a pair $(E, V)$ where $E$ is a vector bundle of rank $r$ and degree $d$, and $V \subseteq H^0 (C, E)$ is a subspace of dimension $n$. In this note we shall investigate the relation between different notions of stability of such systems, and apply our observations to the study of Butler's well known conjecture \cite{But97}. To set the scene, let us recall the notions of $\alpha$-stability and linear stability.

\subsection{\texorpdfstring{$\alpha$}{alpha}-stability of coherent systems}

Let $(E, V)$ be a coherent system over $C$. For a nonnegative real number $\alpha$, the \textsl{$\alpha$-slope} of $(E, V)$ is defined as
\[
\mu_\alpha (E, V) \ := \ \frac{\deg E + \alpha \cdot \dim V}{\rank E} .
\]
(Notice that for $\alpha = 0$, the $\alpha$-slope is simply the slope $\mu(E) = \frac{\deg E}{\rank E}$ of the vector bundle $E$.) Such an $(E, V)$ is \textsl{$\alpha$-semistable} if for all proper coherent subsystems $(F, W)$ we have $\mu_\alpha (F, W) \le \mu_\alpha (E, V)$, and \textsl{$\alpha$-stable} if inequality is strict for all $(F, W)$.

By \cite{KN}, for $\alpha > 0$ there is a moduli space $G(r, d, n; \alpha)$ parametrizing $\alpha$-semistable coherent systems of type $(r, d, n)$ over $C$. Coherent systems and their moduli are the subject of much interest. See \cite{New} for a recent survey of results and open questions.

As discussed in \cite[{\S} 2]{BGMN}, the positive real line is divided into a finite number of intervals separated by \textsl{critical values} $\alpha_1 , \ldots , \alpha_l$ with the property that the moduli spaces $G(r, d, n; \alpha)$ are mutually isomorphic for $0 < \alpha < \alpha_1$, for $\alpha_i < \alpha < \alpha_{i+1}$ and for $\alpha > \alpha_l$. We shall be interested primarily in the first and last intervals. Therefore, once and for all, we choose $\alps \in (0, \alpha_1)$ and $\alpl > \alpha_l$ and consider $\alps$- and $\alpl$-stability.

\subsection{Linear stability}

Mumford introduced linear stability in \cite[Definition 2.16]{Mum} as a property of a subvariety $X \subset \PP^{n-1}$. When $X$ is the curve $C$, linear stability has a useful alternative formulation in terms of the linear series inducing the morphism $C \to \PP^{n-1}$. This was generalized in \cite{CHL} to a notion of linear stability for coherent systems of higher rank over $C$. In fact we shall use a slightly weaker notion than that in \cite{CHL}, which we now discuss.

Recall that a coherent system $(E, V)$ over $C$ is said to be \textsl{generated} if the evaluation map $V \to E|_x$ is surjective for all $x \in C$. Let $(E, V)$ be a generated coherent system of type $(r, d, n)$ over $C$ such that $d > 0$; equivalently, $n > r$. For convenience, we introduce the following function, which might be called ``linear slope''.

\begin{definition} \label{LinearSlope}
For any generated coherent system $(F, W)$ with $\dim W > \rank F$, we set
\begin{equation} \label{lambdaDefn}
\lambda (F, W) \ := \ \frac{\deg F}{\dim W - \rank F} .
\end{equation}
Throughout, we abbreviate $\lambda (F, H^0 (C, F))$ to $\lambda (F)$.
\end{definition}

\noindent The following is a slight modification of \cite[Definition 5.1]{CHL}; the reason for this change is discussed in Remark \ref{WhyWeaker}.

\begin{definition} \label{DefnLinSt}
Let $(E, V)$ be a generated coherent system of type $(r, d, n)$ with $d > 0$. Then $(E, V)$ is said to be \textsl{linearly semistable} if for all generated subsystems $(F, W)$ such with $\deg F > 0$ and such that $E/F$ is not a trivial bundle, we have
\begin{equation} \label{LinStIneq}
\lambda (F, W) \ \ge \ \lambda (E, V) .
\end{equation}
(Note that $F$ may be nonsaturated and/or of full rank.) We say that $(E, V)$ is \textsl{linearly stable} if inequality is strict for all such $(F, W)$.

Clearly it is sufficient to require the inequality (\ref{LinStIneq}) for \textsl{complete} generated subsystems with the stated properties; that is, subsystems of the form $\left( F, V \cap H^0 (C, F) \right)$. 
\end{definition}

A remarkable feature of linear stability is \cite[Theorem 4.12]{Mum}, which states that a linearly stable curve is Chow stable. This has found application in \cite{BT} and elsewhere. Mumford proved a partial analogue for higher dimensional varieties in \cite[Proposition 2.17]{Mum}, but the implications for higher dimensional varieties are far from fully explored. The connection with Definition \ref{DefnLinSt} is as follows: For $(E, V)$ as above, $V$ is naturally identified with a subspace of $H^0 (\PP E^\vee, \opeo)$. It can be shown that Definition \ref{DefnLinSt} is equivalent to linear stability in the sense of \cite{Mum} of the image of $\PP E^\vee \to \PP V^\vee$. A proof of this, together with an exploration of other geometric implications of Definition \ref{DefnLinSt}, will be given in the forthcoming \cite{CH}.

\subsection{Butler's conjecture}

Another significant motivation for studying linear stability in the above sense is the aforementioned conjecture of Butler \cite{But97}, which we now recall. If $(E, V)$ is a generated coherent system, the evaluation sequence
\[
0 \ \to \ \mev \ \to \ \Oc \otimes V \ \to \ E \ \to \ 0
\]
is an exact sequence of vector bundles. The bundle $\mevd$ is called the \textsl{dual span bundle} (hereafter ``DSB'') of $(E, V)$. Following \cite{BMNO}, for $\alpha$ close to $0$, we set
\[
S_0 (r, d, n) \ := \ \left\{ (E, V) \in G(r, d, n; \alpha) : (E, V) \hbox{ generated} \right\} .
\]
Butler's conjecture \cite[Conjecture 2]{But97} can be formulated as follows.

\begin{conjecture} \label{ButlerConj}
Suppose that $C$ is a general curve and $(E, V)$ a general element of any component of $S_0 (r, d, n)$. Then the coherent system $\left( \mevd , V^\vee \right)$ is $\alpha$-stable for $\alpha$ close to $0$, and the map $(E, V) \mapsto \left( \mevd , V^\vee \right)$ is a birational equivalence $S_0 (r, d, n) \dashrightarrow S_0 (n-r, d, n)$.
\end{conjecture}

\begin{definition} \label{ButlerNontriv}
Following \cite{BMNO}, we shall say that \textsl{Butler's conjecture holds nontrivially for type $(r, d, n)$ over $C$} if $S_0 (r, d, n)$ is nonempty and Conjecture \ref{ButlerConj} holds.
\end{definition}

Butler's conjecture has been proven in many cases; see in particular \cite{BBN} for type $(1, d, n)$ and \cite{BMNO} for higher rank. The relevance of linear stability is the following; cf.\ \cite[Lemma 5.4]{CHL}. Suppose that $(F, W)$ is a generated subsystem of $(E, V)$ with $E/F$ nontrivial. Then there is a diagram
\[ \xymatrix{
0 \ar[r] & \mev \ar[r] & \Oc \otimes V \ar[r] & E \ar[r] & 0 \\
0 \ar[r] & M_{W, F} \ar[r] \ar[u] & \Oc \otimes W \ar[r] \ar[u] & F \ar[r] \ar[u] & 0 .
} \]
We deduce in particular that $\mu \left( M_{F, W} \right) = \frac{-\deg F}{\dim W - \rank F} = -\lambda ( F, W )$. Therefore,
\begin{multline} \label{LinStNec}
 \hbox{$\mev$ is a slope (semi)stable vector bundle only if} \\
 \hbox{$(E, V)$ is linearly (semi)stable in the sense of Definition \ref{DefnLinSt}.}
\end{multline}
The converse does not always hold: Examples of linearly stable systems with unstable DSB are given in \cite[{\S} 8]{MS}, \cite[{\S} 4]{CT} and \cite[{\S\S} 4 \& 5]{CHL}. There are, however, many interesting situations in which linear stability does imply slope stability of $\mev$. See in particular \cite{MS} for the case $r = 1$, and \cite{CHL} for further discussion.

\begin{remark} \label{SpecialCurve}
We believe that linear stability may shed light on another aspect of Butler's conjecture: In light of the coherent systems over rational, elliptic and bielliptic curves studied in \cite[{\S} 5]{CHL}, it may be reasonable to expect that linear stability is feasible to check for certain special curves. Thus linear stability, together with sufficient conditions for stability of the DSB such as \cite[Proposition 5.10]{CHL}, may be suited to checking the validity of the conjecture over special curves, and making precise the meaning of ``general'' in the statement.
\end{remark}

\begin{remark} \label{WhyWeaker}
Definition \ref{DefnLinSt} is slightly weaker than \cite[Definition 5.1]{CHL} in that one does not require the linear slope inequality for subsystems $(F, W)$ such that $E/F$ is trivial. This is because if $E/F \cong \Oc^{\oplus r'}$, then using generatedness of $E$ one can check that $E \cong F \oplus \Oc^{\oplus r'}$ and that $M_{F, V \cap H^0 (F)} \cong \mev$. In particular, $M_{F, V \cap H^0 (F)}$ does not slope destabilize $\mev$. In view of (\ref{LinStNec}), we would thus prefer that the subsystem $(F, W)$ should not be regarded as linearly destabilizing $(E, V)$. Therefore, we require the linear slope inequality only for $(F, W)$ where $E/F$ is not trivial.
\end{remark}

Our primary ambition in the present work is to ascertain which implications hold between $\alps$-, $\alpl$- and linear stability for generated coherent systems, and which do not hold. The motivation is to improve understanding of the geometric and moduli-theoretic implications of linear stability, and to give a more complete foundation for approaching Butler's conjecture via linear stability. Although $\alpl$-stability is a priori not an ingredient in Butler's conjecture in the way that $\alps$- and linear stability are, all implications between these conditions are potentially of use in studying the conjecture; and indeed $\alpl$-stability is central in the special case studied in Proposition \ref{MainC} below. Our conclusion, however, is that without further assumptions, the relation between the various stabilities is almost as unpredictable as possible. Our main results are the following.

\begingroup
\setcounter{tmp}{\value{theorem}}
\setcounter{theorem}{0}
\renewcommand\thetheorem{\Alph{theorem}}

\begin{proposition}[Proposition \ref{OneImpl}] \label{MainA}
If a generated coherent system $(E, V)$ is linearly stable and $\alps$-stable, then it is $\alpl$-stable.
\end{proposition}

\begin{theorem} \label{MainB}
There exist generated coherent systems exhibiting all combinations of $\alps$-, $\alpl$- and linear stability which are not ruled out by Proposition \ref{MainA}.
\end{theorem}

Theorem \ref{MainB} is proven by constructing examples with each possible combination of stabilities, using extensions and elementary transformations. In fact, all of these exist already for $r = 2$. Here is an overview:

\begin{equation} \label{Overview}
\begin{tabular}{|c|c|c|c|}
\hline
$\alps$-stable & $\alpl$-stable & linearly stable & \\
\hline
 N & N & N & Example \ref{NNN} \\
 Y & Y & N & Example \ref{YYN} \\
 Y & N & N & Example \ref{YNN} \\
 N & Y & N & Example \ref{NYN} \\
 N & N & Y & Example \ref{NNY} \\
 N & Y & Y & Example \ref{NYY} \\
 Y & Y & Y & Example \ref{YYY} \\
 Y & N & Y & Impossible by Proposition \ref{MainA} \\
\hline
\end{tabular}
\end{equation}

To construct these examples, we rely on the following intuition. In a generated coherent system $(E, V)$ with a subsystem $(F, W)$, very roughly,
\begin{itemize}
\item for $\alps$-stability, $\frac{\deg F}{\rank F}$ should be bounded above;
\item for $\alpl$-stability, $\frac{\dim W}{\rank F}$ should be bounded above; and
\item for linear stability, $\frac{\deg F}{\dim W}$ should be bounded below if $(F, W)$ is generated.
\end{itemize}
See also Remark \ref{DifferentDestab}.

Theorem \ref{MainB} shows the combinations of stabilities which can arise without further assumptions. In special cases, however, further implications may hold between the various stabilities and also stability of the DSB:

\begin{proposition}[Proposition \ref{2d5}] \label{MainC}
Let $C$ be a general curve and $(E, V)$ a generated coherent system over of type $(2, d, 5)$ over $C$ where $d < 3 \cdot d_1 (C) = 3 \cdot \left\lceil \frac{g}{2} + 1 \right\rceil$. Then if $(E, V)$ is $\alpl$-stable, $\mev$ is a stable bundle. In particular, $(E, V)$ is linearly stable.
\end{proposition}

\noindent (For $k \ge 1$, the number $d_k (C)$ is the $k$th gonality of $C$; see {\S} \ref{GonalitySeq} for definitions.) In addition to proving linear stability in Examples \ref{NYY} and \ref{YYY}, Proposition \ref{MainC} allows us to prove a case of Butler's conjecture:

\begin{theorem} \label{MainD}
Let $C$ be a curve of genus $g \ge 18$ which is general in moduli. 
 Then Butler's conjecture holds nontrivially for type $(2, d, 5)$ in the following cases.
\begin{enumerate}
\renewcommand{\labelenumi}{(\alph{enumi})}
\item $d = 2 d_2 (C) - 1$ and $g \equiv 2 \mod 3$
\item $d = 2 d_2 (C)$
\end{enumerate}
\end{theorem}
\endgroup

The paper is organized as follows: After recalling or proving some facts on Brill--Noether loci and secant loci in {\S} \ref{BNDefns}, we prove Proposition \ref{MainA} in {\S} \ref{comparing}. We then give a straightforward construction of a generated rank two extension with many sections, which will furnish most of the examples in the proof of Theorem \ref{MainB}. In {\S} \ref{NonLinStExamples} we construct those examples in (\ref{Overview}) which are not linearly semistable, and thereby in view of (\ref{LinStNec}) cannot have stable DSB; and then proceed to those which are linearly stable in {\S} \ref{LinStExamples}. We prove Theorem \ref{MainD} in the final section.

All the linearly stable systems we construct in {\S} \ref{LinStExamples}, including those which are $\alpha$-nonstable for one or more $\alpha$, turn out to have stable DSB also. It is not surprising that such systems should exist, but it may be useful to have concrete examples.

A recurring feature in the examples is that linear stability and stability of the DSB appear easier to prove for small values of $d$, when the gonality sequence and higher Clifford indices of the curve rule out certain destabilizing subsystems. It is conceivable that other implications between the three types of stability do hold for higher values of $d$. It also remains to be seen what may be possible when a system is strictly semistable with respect to one or more of the definitions.

We also point out that all of the coherent systems we construct are complete and, Remark \ref{SpecialCurve} notwithstanding, live over curves which are general in moduli. Examples of incomplete linearly stable coherent systems (both $\alps$-stable and $\alps$-unstable) are constructed in \cite[{\S\S} 4 \& 5.4]{CHL}, where also the case of bielliptic curves is studied.

We hope that these examples and techniques may be useful in the further study of Butler's conjecture and in the geometric implications of linear stability.

\subsection*{Acknowledgements}

This paper is dedicated to the memory of Professor Peter E.\ Newstead, in fond and deep appreciation for his generosity and wisdom over many years. We also thank Ali Bajravani for enjoyable and useful conversations about linear stability and Clifford indices.

The first author is supported by project IN100723, ``Curvas, Sistemas lineales en superficies proyectivas y fibrados vectoriales'' from DGAPA, UNAM.

\subsubsection*{Notation}

For any sheaf $F$ over $C$ and for each $i \ge 0$, we abbreviate $H^i (C, E)$, $h^i (C, E)$ and $\chi(C, E)$ to $H^i (E)$, $h^i (E)$ and $\chi(E)$ respectively. Moreover, as we shall deal chiefly with complete coherent systems, for any $\alpha \ge 0$ and any sheaf $F$ on $C$ we abbreviate $\mu_\alpha \left( F, H^0 (F) \right)$ to $\mu_\alpha (F)$; and $\lambda (F, H^0 (F))$ to $\lambda (F)$ when $F$ is locally free and generated of positive degree.

\section{Preliminaries}

Let $C$ be a complex projective smooth curve of genus $g \ge 2$. We now recall some notions and results on Brill--Noether theory, Clifford indices and secant loci associated to bundles over $C$ which we shall require.

\subsection{Brill--Noether loci on Picard varieties} \label{BNDefns}

The canonical reference for the following is \cite{ACGH}. For $d \ge 0$ and $k \ge 0$, we consider the Brill--Noether locus 
\[
W_d^k \ = \ \{ L \in \Pic^d (C) : h^0 (L) \ge k + 1 \} ,
\]
which may also be denoted $W_d^k (C)$ or $B_{1, d}^{k+1}$. This is a determinantal variety with expected dimension given by the Brill--Noether number
\[
\beta^{k + 1}_{1, d} \ = \ g - (k + 1) (k - d + g) .
\]
The locus $W_d^k$ is nonempty whenever $\beta^{k + 1}_{1, d}$ is nonnegative. If $C$ is general in moduli, $W_d^k$ is nonempty and of dimension $\beta^{k + 1}_{1, d}$ whenever $0 \le \beta^{k + 1}_{1, d} \le g$.

\subsubsection{The gonality sequence of a curve} \label{GonalitySeq}

A reference for this is \cite[{\S} 4]{LNe}. For a fixed curve $C$ of genus $g$ and for $k \ge 1$ we define
\[
d_k (C) \ := \ \min \left\{ \deg L : L \to C \hbox{ with } h^0 (L) \ge k + 1 \right\} .
\]
A computation using {\S} \ref{BNDefns} shows that
\begin{equation} \label{dkC}
d_k (C) \ \le \ \left\lceil \frac{kg}{k+1} + k \right\rceil ,
\end{equation}
with equality when $C$ is general; for example, Petri.

\subsubsection{Base point freeness of a general element}

We shall use the following statement several times.

\begin{lemma} \label{BasePointFree}
Let $C$ be a general curve of genus $g > 2$ and let $k \ge 1$ and $\ell$ be integers satisfying $\ell \ge d_k (C)$ and $k - \ell + g \ge 0$. Then a general $L \in W^k_\ell (C)$ satisfies $h^0 (L) = k+1$ and is generated.
\end{lemma}

\begin{proof}
Using {\S} \ref{BNDefns} and the hypothesis that $k - \ell + g \ge 0$, we have
\begin{multline}
\beta^{k+1}_{1, \ell} - \beta^{k+2}_{1, \ell} \ = \ g - (k + 1)(k + 1 - \ell + g - 1) - g + (k + 2)(k + 2 - \ell + g - 1) \\
 = \ (k + 2)(k + 1 - \ell + g) - (k+1)(k + 1 - \ell + g) + (k + 1) \\
 = \ (k + 1 - \ell + g) + (k + 1) \ \ge \ k + 2 . \label{BNnumDiff}
\end{multline}
As $C$ is general, then, $W^{k+1}_\ell$ is not dense in $W^k_\ell$; whence $h^0 (L) = k+1$ for general $L \in W^k_\ell$. For the rest: If $L \in W^k_\ell$ and $p$ is a base point of $|L|$, then $L(-p) \in W^k_{\ell - 1}$, and $L$ belongs to the image of the map $W^k_{\ell - 1} \times C \to W^k_\ell$ given by $(L', p) \mapsto L'(p)$. By generality of $C$, the locus of such $L$ has dimension at most $\beta^{k+1}_{1, \ell - 1} + 1$. A calculation similar to that in (\ref{BNnumDiff}) shows that this dimension is strictly less than $\beta^{k+1}_{1, \ell} = \dim W^k_\ell$. 
 It follows that a general $L \in W^k_\ell$ is base point free.
\end{proof}

\subsection{Higher rank Clifford indices}

We recall now some facts and results on Clifford indices of vector bundles over $C$, which will be used in constructing linearly stable systems in {\S} \ref{LinStExamples}. A reference for the following is \cite{LNe}.

\begin{definition} \label{DefnGammaR}
Let $C$ be a curve of genus $g \ge 4$. For any vector bundle $E \to C$, we define
\[
\gamma(E) \ := \ \frac{1}{r} \left( d - 2 ( h^0 (E) - r ) \right) \ = \ \mu(E) - 2 \cdot \frac{h^0 (E)}{r} + 2 .
\]
The \emph{$r$th Clifford index of $C$} which we shall use is
\[
\gamma_r' (C) \ := \ \min \left\{ \gamma (E) : E \to C \hbox{ semistable of rank $r$ with $\mu(E) \le g - 1$ and $h^0 (E) \ge 2r$} \right\} .
\]
\end{definition}

It is well known that if $C$ is general in moduli, then
\begin{equation} \label{CliffordIndexGeneral}
\gamma_1' (C) \ = \ d_1 (C) - 2 \ = \ \left\lceil \frac{g}{2} - 1 \right\rceil .
\end{equation}
For any $C$, taking sums of line bundles, one sees that $\gamma_r' (C) \le \gamma_1' (C)$. The Mercat conjecture (see \cite[{\S} 3]{LNe}) would imply that equality obtains. This is now known to be false for special curves; see for example \cite[Theorem 1.1]{FO}. However, we shall use the following important result of Bakker and Farkas \cite{BF} for rank two bundles over general curves.

\begin{theorem}[\cite{BF}, Theorem 1] \label{MercatRankTwo}
Let $C$ be a curve of genus $g \ge 4$ which is general in moduli. Then $\gamma_2'(C) = \gamma_1' (C) = d_1 (C) - 2$.
\end{theorem}

\subsection{Secant loci and generalized secant loci} \label{GenSec}

We shall use the following in Proposition \ref{PlaneCurveSing} and Example \ref{NYN}. Let $L$ be a line bundle such that $\Kc L^{-1}$ has nonempty linear system. Denote by $\Ce$ the $e$th symmetric product of $C$. Generalizing the varieties $C^r_d$ studied in \cite[Chapter IV]{ACGH}, for $e \ge f \ge 0$ we consider the \textsl{secant locus}
\begin{equation} \label{VefKLiDefn}
\Vef (\Kc L^{-1}) \ := \ \left\{ D \in \Ce : h^0 ( \Kc L^{-1}(-D)) \ge h^0 (\Kc L^{-1}) - e + f \right\} .
\end{equation}
This is a determinantal subvariety of $\Ce$, of expected dimension $e - f \cdot ( h^0 (M) - e + f )$. These varieties have been studied in \cite{AS}, \cite{Baj} and in very many other works\footnote{Note that $\Vef (\Kc L^{-1})$ is denoted by $V_e^f ( \Kc L^{-1})$ in \cite{Baj}.}. By Riemann--Roch and Serre duality, we have also
\begin{equation} \label{VefLDefn}
\Vef (\Kc L^{-1}) \ = \ \left\{ D \in \Ce : h^0 (L(D)) \ge h^0 (L) + f \right\} .
\end{equation}

We consider also the following direct generalization to higher rank. Let $S$ be a locally free sheaf of rank $r$ over $C$ with $h^0 (\Kc \otimes S^\vee ) \ge 1$. For $e \ge 1$, let $\Quot^e (S^\vee)$ be the Quot scheme parametrizing torsion quotients of length $e$ of $S^\vee$; equivalently, locally free subsheaves of full rank and degree $\deg S - e$. The following \textsl{generalized secant locus} is a special case of \cite[Definition 2.17]{Hit}:
\begin{multline*}
\Qef ( S, \Kc, H^0 (\Kc \otimes S^\vee) ) \ = \ \left\{ \left[ S^\vee \subset F^\vee \right] \in \Quot^e (S^\vee) : \right. \\
\left. h^0 (\Kc \otimes F^\vee) \ge h^0 ( \Kc \otimes S^\vee ) - e + f \right\}
\end{multline*}
This is a determinantal subvariety of $\Quot^e (S^\vee)$, of expected dimension
\[
re - f \cdot ( h^0 (\Kc \otimes S^\vee) - e + f ) .
\]
As above, $\Qef ( S, \Kc, H^0 (\Kc \otimes S^\vee) )$ may be realized as
\begin{equation} \label{GenSecRkR}
\left\{ \left[ S^\vee \subset F^\vee \right] \in \Quot^e (S^\vee) : h^0 ( F ) \ge h^0 ( S ) + f \right\} .
\end{equation}
In the sequel, we shall most often be concerned with the case $f = 1$.

\section{Comparing linear stability and \texorpdfstring{$\alpha$}{alpha}-stability} \label{comparing}

Let $\alps$ and $\alpl$ be as defined in the introduction. We shall investigate which combinations of $\alps$-, $\alpl$- and linear stability can occur for a generated coherent system over $C$. On a related note, Butler discusses in \cite{But97} the relation between $\alps$-stability of a coherent system and slope stability of the ambient vector bundle. We sum up this discussion in the following useful lemma.

\begin{lemma} \label{StVBImpliesAlps} Let $C$ be a curve and $(E, V)$ a coherent system over $C$.
\begin{enumerate}
\item[(a)] If $E$ is a stable vector bundle, then $(E, V)$ is $\alps$-stable.
\item[(b)] If $E$ is a nonsemistable vector bundle, then $(E, V)$ is not $\alps$-semistable.
\item[(c)] If $(E, V)$ is $\alps$-stable, then $E$ is a semistable vector bundle.
\end{enumerate}
\end{lemma}

\begin{proof}
This is implicit in \cite[p.\ 3]{But97}; we give the details for the reader's convenience. If $(F, W) \subset (E, V)$ is a proper subsystem, we write $(r_F, d_F, n_F)$ for the type of $(F, W)$.

(a) Let $(F, W) \subset (E, V)$ be a proper subsystem. We wish to show that if $\alpha$ is close to zero, then we have
\begin{equation} \label{AlphaSlopeIneq}
\alpha \cdot \left( \frac{n_F}{r_F} - \frac{n}{r} \right) \ < \ \mu (E) - \mu (F).
\end{equation}
Since $E$ is stable, $\mu(E) - \mu(F) > 0$. Thus if $\frac{n_F}{r_F} - \frac{n}{r} \le 0$ then (\ref{AlphaSlopeIneq}) holds for all $\alpha \ge 0$. Otherwise, for any $\alpha$ such that
\[ 0 \ < \ \alpha \ < \ \frac{\mu (E) - \mu(F)}{\frac{n_F}{r_F} - \frac{n}{r}} , \]
we have $\mu_\alpha (F, W) < \mu_\alpha (E, V)$. As there are finitely many possible types $(r_F, d_F, n_F)$ with $\frac{n_F}{r_F} - \frac{n}{r} > 0$, for one $\alpha$ sufficiently close to zero we obtain (\ref{AlphaSlopeIneq}) for all proper subsystems $(F, W)$. Thus $(E, V)$ is $\alps$-stable.

(b) Suppose that $F \subset E$ is a subbundle such that $\mu(F) > \mu (E)$. Then we claim for $\alpha$ close to $0$ that
\[ \mu (F) - \mu (E) \ > \ \alpha \cdot \left( \frac{n}{r} - \frac{n_F}{r_F} \right) . \]
If $\frac{n}{r} - \frac{n_F}{r_F} \le 0$ then this holds for $\alpha \ge 0$; otherwise it holds for any $\alpha$ such that
\[ 0 \ < \ \alpha \ < \ \frac{\mu (F) - \mu(E)}{\frac{n}{r} - \frac{n_F}{r_F}} . \]
For such $\alpha$ we have $\mu_\alpha (F, W) > \mu_\alpha (E, V)$, and therefore $(E, V)$ is not $\alps$-semistable.

(c) Let $(F, W) \subset (E, V)$ be a proper subsystem. It is easy to check that by the $\alpha$-slope inequality for $\alpha$ close to zero,
\begin{equation} \label{PartCineq}
d r_F - d_F r \ > \ \alpha \cdot ( n_F r - n r_F ) .
\end{equation}
If $n_F r - n r_F \ge 0$ then $\mu (E) - \mu(F) > 0$. Otherwise, choose $\alpha$ such that $0 < \alpha < \frac{1}{-( n_F r - n r_F )}$. Then it follows from (\ref{PartCineq}) that $d r_F - d_F r$ is an integer greater than $-1$, and so we obtain $\mu (E) - \mu(F) \ge 0$. Thus $E$ is semistable.
\end{proof}

\noindent Let us now prove the only implication holding between our three types of stability.

\begin{proposition} \label{OneImpl}
Suppose that $(E, V)$ is a generated coherent system of type $(r, d, n)$ with $d > 0$ and $r \ge 2$ which is $\alps$-stable and linearly stable. Then $(E, V)$ is $\alpl$-stable.
\end{proposition}

\begin{proof}
Let $(F, W)$ be a coherent subsystem of $(E, V)$. We wish to show that
\begin{equation} \label{goal}
\frac{\deg F}{\rank F} + \alpl \cdot \frac{\dim W}{\rank F} \ < \ \frac{d}{r} + \alpl \cdot \frac{n}{r} .
\end{equation}
Let $F_W$ be the subsheaf of $F$ generated by $W$. We dispose firstly of a special case. Suppose that $\deg F_W = 0$; equivalently, that $F_W$ is a trivial sheaf. Then we have
\[
0 \ \le \ \dim W \ = \ \rank F_W \ \le \ \rank F ,
\]
and it follows that
\[
\mu_\alpl (F, W) \ = \ \mu(F) + \alpl \cdot \frac{\dim W}{\rank F} \ \le \ \mu(F) + \alpl .
\]
As $d > 0$ and $(E, V)$ is generated, $\frac{n}{r} > 1$. Thus, after increasing $\alpl$ if necessary, we obtain
\[
\mu_\alpl (E, V) \ = \ \frac{d}{r} + \alpl \cdot \frac{n}{r} \ > \ \mu(F) + \alpl \ \ge \ \mu_\alpl (F, W) .
\]
Thus we may assume that $\deg F_W > 0$.

Now since we have supposed $(E, V)$ to be $\alps$-stable, $E$ is a semistable vector bundle by Lemma \ref{StVBImpliesAlps} (c); and so $\frac{\deg F}{\rank F} \le \frac{d}{r}$. Therefore, the slope inequality (\ref{goal}) would follow if we could show that
\[
\frac{\dim W}{\rank F} \ < \ \frac{n}{r} .
\]
Since $\rank F \ge \rank F_W$, also $\frac{\dim W}{\rank F} \le \frac{\dim W}{\rank F_W}$. Thus in fact it would suffice to prove that
\begin{equation} \label{goalPrime}
\frac{\dim W}{\rank F_W} \ < \ \frac{n}{r} .
\end{equation}

Now since $E$ is semistable, we have also $\frac{\deg F_W}{\rank F_W} \le \frac{d}{r}$. As $F_W$ is generated and of positive degree, $\frac{\dim W}{\rank F_W} - 1 > 0$, and so
\begin{equation} \label{one}
\frac{\deg F_W / \rank F_W}{\frac{\dim W}{\rank F_W} - 1} \ \le \ \frac{d / r}{\frac{\dim W}{\rank F_W} - 1} .
\end{equation}
Furthermore, as $E$ is semistable of positive slope, it has no nonzero trivial quotient. Thus, as $(F_W , W)$ is a generated subsystem of positive degree in $(E, V)$, the hypothesis of linear stability implies that
\begin{equation} \label{two}
\frac{d/r}{n/r - 1} \ = \ \frac{d}{n-r} \ < \ \frac{\deg F_W}{\dim W - \rank F_W} \ = \ \frac{\deg F_W / \rank F_W}{\frac{\dim W}{\rank F_W} - 1} .
\end{equation}
Combining (\ref{one}) and (\ref{two}), we obtain
\[
\frac{d/r}{n/r - 1} \ < \ \frac{d / r}{\frac{\dim W}{\rank F_W} - 1} ,
\]
from which the desired inequality (\ref{goalPrime}) follows easily.
\end{proof}

In what follows, we shall show that all the remaining combinations of stabilities in (\ref{Overview}) can occur. To warm up, we exhibit a coherent system which is neither $\alps$-, $\alpl$- nor linearly semistable.

\begin{example} \label{NNN}
Let $L$ and $M$ be generated line bundles with $\deg M > \deg L > 0$ and $n_1 := h^0 (M) = h^0 (L)$. For any $\alpha \ge 0$, we have
\[
\mu_\alpha (M) \ = \ \deg M + \alpha \cdot h^0 (M) \ > \ \mu (L \oplus M) + \alpha \cdot \frac{h^0 (L \oplus M)}{2} \ = \ \mu_\alpha \left( L \oplus M \right) .
\]
Thus the system is not $\alpha$-semistable for any $\alpha \ge 0$. For the rest: As $\frac{L \oplus M}{L}$ is not trivial and
\[
\lambda (L) \ = \ \frac{\deg L}{n_1 - 1} \ < \ \frac{\deg L + \deg M}{2 \cdot (n_1 - 1)} \ = \ \lambda \left( L \oplus M \right) ,
\]
the system is not linearly semistable. \qed
\end{example}

\begin{remark} \label{DifferentDestab}
Not unexpectedly, this is the easiest of our examples to obtain. It is nonetheless instructive that a linearly desemistabilizing subsystem may not be $\alpha$-destabilizing, and vice versa. This illustrates the intuition after (\ref{Overview}).
\end{remark}

\section{Linearly nonsemistable coherent systems} \label{NonLinStExamples}

The following extension construction will yield several of our examples.

\begin{lemma} \label{mfExt}
Let $L_1$ and $L_2$ be generated line bundles over $C$. For $i \in \{ 1, 2 \}$, write $\deg L_i =: \ell_i$ and $h^0 (L_i) =: k_i$. Suppose that
\begin{equation} \label{AllSectsLift}
\ell_2 \ > \ k_1 k_2 + (k_2 - 1)(g - 1 - \ell_1) .
\end{equation}
Then there exists a nontrivial extension $0 \to L_1 \to E \to L_2 \to 0$ in which all sections of $L_2$ lift to $E$. In particular, the coherent system $(E, H^0(E))$ is of type $(2, \ell_1 + \ell_2, k_1 + k_2 )$ and generated with no nonzero trivial quotient.
\end{lemma}

\begin{proof}
Let $0 \to L_1 \to E \to L_2 \to 0$ be an extension with class $\delta_E \in H^1 (L_2^{-1} L_1 )$. The condition that all sections of $L_2$ lift to $E$ is equivalent to saying that the cup product map $\cup \delta_E \colon H^0 (L_2) \to H^1 (L_1)$ is zero. Thus there exists a nontrivial extension with this property if and only if
\[
\Ker \left( \cup \colon H^1 (L_2^{-1} L_1) \ \to \ \Hom \left( H^0 (L_2), H^1 (L_1) \right) \right)
\]
is nonzero. Clearly this is the case if $h^1 (L_2^{-1} L_1) > h^0 (L_2) \cdot h^1 (L_1)$. Let us check that this follows from the inequality (\ref{AllSectsLift}). By Riemann--Roch, we have
\[
h^1 (L_2^{-1} L_1) \ \ge \ -\chi (L_2^{-1} L_1) \ = \ 
 \ell_2 + ( g - 1 - \ell_1 ) ,
\]
and $h^1 (L_1) = 
k_1 + ( g - 1 - \ell_1 )$. Thus $h^1 (L_2^{-1} L_1) > h^0 (L_2) \cdot h^1 (L_1)$ if
\[
\ell_2 + ( g - 1 - \ell_1 ) \ > \ k_2 k_1 + k_2 ( g - 1 - \ell_1 ) ,
\]
that is, $\ell_2 > k_2 k_1 + (k_2 - 1) (g - 1 - \ell_1)$, as desired. To finish: Clearly each $h^0 (L_i^\vee) = 0$, and so $E$ has no nonzero trivial quotient.
\end{proof}

We now use Lemma \ref{mfExt} to construct a generated coherent system which is $\alpha$-stable for all $\alpha \ge 0$, but not linearly semistable.

\begin{example} \label{YYN}
Let $C$ be a general curve of genus $g \ge 6$. Firstly, we claim that we may choose an integer $\ell$ satisfying $d_1 (C) < \ell \le g$, where $d_1 (C) = \left\lceil \frac{g}{2} + 1 \right\rceil$ is as defined in {\S} \ref{GonalitySeq}. Such an $\ell$ exists if there is an integer between $g$ and $\left\lceil \frac{g}{2} + 1 \right\rceil$. This is certainly the case if
\[
g - \left( \frac{g}{2} + 1 \right) \ \ge \ 2 ;
\]
which follows from the hypothesis that $g \ge 6$.

Thus we may choose $L_1 \in W_{\ell}^1$ and $L_2 \in W_{\ell + 1}^1$. Perturbing if necessary, by Lemma \ref{BasePointFree} we may assume that $h^0 (L_1) = 2 = h^0 (L_2)$ and that both the $L_i$ are generated. 
 Now set
\[
k_1 \ = \ k_2 \ = \ 2 \quad \ \hbox{and} \quad \ell_1 \ = \ \ell \quad \hbox{and} \quad \ell_2 \ = \ \ell + 1 ,
\]
The inequality (\ref{AllSectsLift}) now becomes 
\[
\ell + 1 \ > \ 2 \cdot 2 + (2 - 1) (g - 1 - \ell) ;
\]
that is, $\ell > \frac{g}{2} + 1$, which follows by hypothesis. Thus by Lemma \ref{mfExt}, there exists an extension $0 \to L_1 \to E \to L_2 \to 0$ such that $(E, H^0 (E))$ is generated of type $(2, 2\ell + 1, 4)$. Let $N \subset E$ be a line subbundle. As the extension is nontrivial, it is easy to see that $\deg N \le \ell$ and $h^0 (N) \le 2$. Thus
\[
\mu_\alpha ( N ) \ \le \ \ell + \alpha \cdot 2 \ < \ \frac{2 \ell + 1}{2} + \alpha \cdot \frac{4}{2} \ = \ \mu_{\alpha} (E)
\]
for any $\alpha \ge 0$, and in particular $E$ is both $\alps$- and $\alpl$-stable. However,
\[
\lambda (L_1) \ = \ \frac{\ell}{h^0 (L_1) - 1} \ = \ \ell \ < \ \ell + \frac{1}{2} \ = \ \frac{2 \ell + 1}{h^0 (E) - 2} \ = \ \lambda(E) ,
\]
where $\lambda$ is the linear slope function (\ref{LinearSlope}). As $E/L_1$ is not trivial, $(E, H^0 (E))$ is not linearly semistable. \qed
\end{example}

\begin{remark} \label{subpencil}
In the terminology of \cite{BMNO}, the system $(E, H^0 (E))$ above \emph{contains a pencil}; that is, a subsystem of type $(1, \ell, 2)$. In general, if $n \le 2r$ and $(E, V)$ is generated of type $(r, d, n)$ with semistable ambient bundle, then
\[
\frac{\deg L}{\dim \left( V \cap H^0 (L) \right) - 1} \ \le \ \deg L \ \le \ \frac{d}{r} \ \le \ \frac{d}{n - r} .
\]
It is then easy to see that any pencil linearly destabilizes $(E, V)$. By (\ref{LinStNec}), it also prevents $\mev$ from being stable. The role of pencils in destabilizing dual span bundles is discussed in \cite[{\S\S} 4--5]{BMNO}.
\end{remark}

In analogy with the notion of containing a pencil discussed above, we make the following definition.

\begin{definition} \label{subnet}
Let $(E, V)$ be a coherent system of type $(r, d, n)$ where $n \ge 3$. We say that \textsl{$(E, V)$ contains a net} if $(E, V)$ has a coherent subsystem of type $(1, e, 3)$.
\end{definition}

We now use this notion to give a characterization of $\alpl$-stability of systems of type $(2, d, 5)$ which will be used several times.

\begin{lemma} \label{SpecialImpl}
Let $(E, V)$ be a coherent system of type $(2, d, 5)$. If $(E, V)$ contains a net, then it is not $\alpl$-semistable. If $(E, V)$ does not contain a net, then it is $\alpl$-stable.
\end{lemma}

\begin{proof}
If $(E, V)$ contains a net $(L, W)$ then it is easy to check that $(E, V)$ is not $\alpha$-semistable for $\alpha > d - 2 \cdot \deg L$.
 Conversely, if $(E, V)$ does not contain a net, then $\mu_\alpha (L, W) \le \deg L + 2 \alpha$ for any rank one subsystem $(L, W)$. If $e$ is the maximal degree attained by line subbundles of $E$, then for $\alpha > 2e - d$ we have
\[
\mu_\alpha (L, W) \ \le \ e + 2 \alpha \ < \ \mu(E) + \frac{5 \alpha}{2} \ = \ \mu_\alpha (E, V) .
\]
Therefore, $(E, V)$ is $\alpl$-stable.
\end{proof}

We now construct a system which is $\alps$-stable, but neither $\alpl$- nor linearly semistable.

\begin{example} \label{YNN}
Let $C$ be a general curve of genus $g \ge 12$. By a similar computation to that in Example \ref{YYN}, we see that we may choose an integer $\ell$ satisfying $d_2 (C) < \ell \le g$, where $d_2 (C) = \left\lceil \frac{2g}{3} + 2 \right\rceil$. 
 Let $L_1$ and $L_2$ be general elements of $W_{\ell}^2$ and $W_{\ell + 1}^1$. By Lemma \ref{BasePointFree}, we may assume that $h^0 (L_1) = 3$ and $h^0 (L_2) = 2$, and that both the $L_i$ are generated. 
 Now we set
\[
k_1 \ = \ 3 \quad \hbox{and} \quad k_2 \ = \ 2 , \quad \ \hbox{and as before} \quad \ell_1 \ = \ \ell \quad \hbox{and} \quad \ell_2 \ = \ \ell + 1 .
\]
Computing as in the previous example, we see that the inequality (\ref{AllSectsLift}) in this case follows from $\ell > d_2 (C)$. 
 Thus by Lemma \ref{mfExt}, there exists an extension $0 \to L_1 \to E \to L_2 \to 0$ such that $(E, H^0 (E))$ is generated of type $(2, 2\ell + 1, 5)$. As before, $\deg N \le \ell$ for all line subbundles $N \subset E$, and so $E$ is a stable vector bundle. By Lemma \ref{StVBImpliesAlps} (a), the system $(E, H^0 (E))$ is $\alps$-stable. However, as $(E, H^0 (E))$ contains the net $(L_1 , H^0 (L_1))$, by Lemma \ref{SpecialImpl} it is not $\alpl$-stable.
 For the rest: Since $(E, H^0 (E))$ is $\alps$-stable but not $\alpl$-stable, by Proposition \ref{OneImpl} it is at best linearly strictly semistable. But in fact
\[
\lambda (L_1) \ = \ \frac{\ell}{2} \ < \ \frac{2\ell + 1}{3} \ = \ \lambda(E) ;
\]
as $E/L_1$ is not trivial, $E$ is not linearly semistable. \qed \end{example}

The last of our linearly nonsemistable systems will be of type $(2, d, 5)$, and $\alpl$-stable but not $\alps$-stable. The construction is somewhat more involved than our other examples. The ambient bundle is constructed as an elementary transformation instead of an extension. Before beginning the construction, we prove the existence of a certain line bundle which we shall require.

For the remainder of {\S} \ref{NonLinStExamples}, we fix the following. Let $C$ be a general curve of genus $g \ge 18$. Set $e := \frac{1}{3} (g - 1 + \varepsilon)$, where $\varepsilon \in \{ 1, 2, 3 \}$ is such that $e$ is an integer. Then
\[
\beta^2_{1, g-e} \ = \ g - 2e - 2 
 \ \ge \ \frac{g-10}{3} \ > \ 0 .
\]
Thus $W^1_{g-e}$ is nonempty. Let $L \in W^1_{g-e}$ be a general element. By Lemma \ref{BasePointFree}, 
 we may assume that $h^0 ( L ) = 2$ and that $L$ is base point free.

\begin{proposition} \label{PlaneCurveSing}
Let $C$ and $L$ be as above. Then there exists an effective divisor $D \in \Ce$ such that $h^0 (L (D)) = 3$ and $L (D)$ is base point free.
\end{proposition}

\begin{proof}
Consider the Abel--Jacobi-type map $\alpha \colon \Ce \to \Pic^g (C)$ given by $D \mapsto L (D)$. Now $L(D') \cong L(D)$ if and only if $\Oc (D') \cong \Oc (D)$. But $e \le \frac{g+2}{3} < \frac{g+2}{2} \le d_1 (C)$ since $C$ is general, and so $\Oc (D') \cong \Oc (D)$ if and only if $D' = D$. Thus $\alpha$ is injective, and its image is a closed sublocus of $\Pic^g (C)$ of dimension $e$.

Now $\beta_{1, g}^3 = g - 6$, so $W^2_g$ has codimension at most $6$ in $\Pic^g (C)$. 
 Moreover, $W^2_g$ is a degeneracy locus of a map $\cE \to \cF$ of vector bundles over $\Pic^g (C)$ where $\Hom (\cE, \cF)$ is ample, by \cite[VII.2]{ACGH}. Thus by \cite[Proposition VII.1.3]{ACGH} it intersects $\Image \alpha$ in codimension at most $6$. As $g \ge 18$, we have $e \ge 6$. Thus there exists a divisor $D \in \Ce$ such that $h^0 (L (D)) \ge 3$.

In particular, the secant locus $\Veo (\Kc L^{-1})$ defined in (\ref{VefLDefn}) is nonempty. By Riemann--Roch, $h^0 (\Kc L^{-1}) = e + 1$. Thus $\Veo (\Kc L^{-1})$ has dimension at least
\[
e - ( h^0 ( \Kc L^{-1} ) - e + 1 ) \ = \ e - ( e + 1 - e + 1 ) \ = \ e - 2 .
\]
By \cite[Lemma 4.1]{Baj} (which one can check remains valid without the hypothesis of very ampleness), no component of $\Veo (\Kc L^{-1})$ is contained in $V^{e-2}_e (\Kc L^{-1})$. Thus we may assume that $h^0 ( L (D) ) = 3$.

Next: Since $L$ is base point free, $L (D)$ fails to be base point free if and only if $D = D' + x$ for some $x \in C$ and $D' \in C_{e-1}$ with $h^0 ( L (D')) = 3$. By (\ref{VefLDefn}), such a $D'$ belongs to $V^{e-2}_{e-1} ( \Kc L^{-1} )$.

Now by \cite[Theorem 3.2]{Baj}, the locus $V^{e-2}_{e-1} ( \Kc L^{-1} )$ is smooth and of expected dimension at $D'$ if and only if the multiplication map
\[
\mu_0^{\Kc L^{-1}} \colon H^0 ( \Oc (D') ) \otimes H^0 ( \Kc L^{-1} (-D') ) \ \to \ H^0 ( \Kc L^{-1} )
\]
is injective. As $C$ is general,
\[
e - 1 \ = \ \frac{1}{3}( g - 4 + \varepsilon ) \ < \ \frac{g}{2} + 1 \ \le \ d_1 (C) ,
\]
whence $h^0 (\Oc (D')) = 1$. It follows easily that $\mu_0^{\Kc L^{-1}}$ is injective, and all components of $V^{e-2}_{e-1} (\Kc L^{-1})$ attain their expected dimension $(e-1) - (h^0 (\Kc L^{-1}) - (e - 1) + 1 ) = e - 4$. Therefore, the locus
\[
\left\{ D' + x : D' \in V^{e-2}_{e-1} (\Kc L^{-1}), x \in C \right\}
\]
has dimension at most $e - 3$. We deduce that if $D \in \Veo (\Kc L^{-1})$ is general, then $L(D)$ is base point free. This completes the proof of the proposition.
\end{proof}

\begin{remark}
Proposition \ref{PlaneCurveSing} implies that a general curve of genus $g \ge 18$ has a plane model of degree $g$ with a point of multiplicity at least $\frac{g}{3}$.
\end{remark}

Now we are in a position to construct a generated coherent system which is $\alpl$-stable but neither $\alps$- nor linearly semistable.

\begin{example} \label{NYN}
Let $C$, $L$ and $D$ be as above, and let $M$ be a general element of $\Pic^{g+1} (C)$. Set $S := L \oplus M$. As $h^0 (L(D) \oplus M) = h^0 (S) + 1$, the short exact sequence
\[
0 \ \to \ S \ \to \ L(D) \oplus M \ \to \ \cO_D \ \to \ 0
\]
defines an element of the generalized secant locus $\Qeo ( S^\vee , \Kc, H^0 (\Kc \otimes S^\vee) )$ defined in (\ref{GenSecRkR}). This is a determinantal sublocus of $\Quot^e ( S^\vee )$ of expected dimension
\[
\dim \Quot^e (S^\vee ) - (h^0 ( \Kc \otimes S^\vee ) - e + 1) 
\ = \ 2e - 2 .
\]
Let $\cQ$ be an irreducible component of $\Qeo \left( S^\vee , \Kc, H^0 (\Kc \otimes S^\vee) \right)$ containing the element $( L ( D ) \oplus M )^\vee$.

Now each element $\left[ F^\vee \subset S^\vee \right]$ of $\cQ$ defines a sequence
\begin{equation} \label{FExactSeq}
0 \ \to \ L \oplus M \ \to \ F \ \to \ \tau \ \to \ 0
\end{equation}
where $\tau$ is torsion of length $e$. We claim that for a general such $F$, the complete coherent system $(F, H^0 (F))$ is generated of type $(2, 2g+1, 5)$ and $\alpl$-stable, but neither $\alps$- nor linearly semistable. Abusing language, we shall refer to $F$ as an element of $\cQ$.

Firstly, $L(D) \oplus M$ is generated with $h^0 ( L(D) \oplus M ) = 5$ by Proposition \ref{PlaneCurveSing}. Since these properties are open in suitable families, the same is true for a general $F \in \cQ$. Let us show that $F$ is $\alpl$-stable. Now since $M$ is saturated in $L(D) \oplus M$, it is saturated inside a general $F \in \cQ$ since torsion freeness is an open condition by \cite[Proposition 2.1]{Mar}. For such an $F$, comparing determinants, we have a short exact sequence of locally free sheaves
\[
0 \ \to \ M \ \to \ F \ \to \ L (\Delta) \ \to \ 0
\]
where $\Delta \in \Ce$ and $h^0 (L (\Delta) ) = 3$ (so in fact $\Delta \in \Veo ( \Kc L^{-1} )$). As $F$ is generated, so is $L(\Delta)$. Thus if $N \subset F$ is a line subbundle with $h^0 (N) \ge 3$, then clearly $N \to F \to L(\Delta)$ is an isomorphism, and $F$ splits as $M \oplus L(\Delta)$. But the family
\[
\left\{ M \oplus L (\Delta) : \Delta \in \Ce \right\}
\]
is of dimension at most $e$, which is less than $2e - 2 \le \dim \cQ$ since $g \ge 18$. Thus a general $F \in \cQ$ has no line subbundle $N$ with $h^0 (N) \ge 3$, and is therefore $\alpl$-stable by Lemma \ref{SpecialImpl}.

For the rest: As
\[
\deg M \ = \ g + 1 \ > \ \frac{\deg L (D) + \deg M}{2} \ = \ \frac{2g + 1}{2} \ = \ \mu(F) ,
\]
the bundle $F$ is not semistable, and so $(F, H^0 (F))$ is not $\alps$-semistable by Lemma \ref{StVBImpliesAlps} (b). Finally, by (\ref{FExactSeq}), any $F \in \cQ$ contains $L$ as a generated subsheaf of degree $g - e > 0$. As $g - e < 2g + 1 = \deg F$, the quotient $F/L$ is not trivial. But
\begin{multline*}
\lambda (L) \ = \ \frac{\deg L}{h^0 (L) - \rank L} \ = \ g - e \ = \ g - \left( \frac{g - 1 + \varepsilon}{3} \right) \ = \\
 \frac{2g + 1 - \varepsilon}{5 - 2} \ < \ \frac{2g + 1}{h^0 (F) - \rank F} \ = \ \lambda (F) .
\end{multline*}
Thus $F$ is not linearly semistable.

In summary, if $F \in \cQ$ is general then $(F, H^0 (F))$ is a coherent system of type $(2, 2g + 1, 5)$ which is $\alpl$-stable, but not $\alps$- or linearly semistable.
\qed \end{example}

\begin{remark}
The secant locus $\Qeo (\Kc \otimes S^\vee)$ contains at least one component in addition to $\cQ$. Any elementary transformation of the form
\[
0 \ \to \ L \oplus M(x) \ \to \ F' \ \to \ \tau' \ \to 0 ,
\]
where $\tau'$ is torsion of degree $e - 1$, defines an element of $\Qeo ( \Kc \otimes S^\vee )$. The locus of such $F'$ has dimension
\[
\dim C + \dim \Quot^{e-1} ( L \oplus M (x) ) \ = \ 1 + 2(e-1) \ = \ \dim \Quot^e ( L \oplus M ) - 1,
\]
which is in fact greater than the expected dimension of $\Qeo ( \Kc \otimes S^\vee )$. If $h^0 (F') = 5$ then, since $h^0 (M(x)) = 3$, by Lemma \ref{SpecialImpl} the system $(F', H^0 (F'))$ is not $\alpl$-stable. However, this locus does not contain $\cQ$ because we have seen that $M$ is saturated inside a general $F \in \cQ$.
\end{remark}

\section{Linearly stable coherent systems} \label{LinStExamples}

The remaining coherent systems which we shall construct are linearly stable. By (\ref{LinStNec}), this is a necessary condition for stability of the DSB. We shall see that all the examples in this section in fact have stable DSB; indeed, this is how we shall prove linear stability in two cases. We continue to use the extension construction of Lemma \ref{mfExt}.

\subsection{A system which is not \texorpdfstring{$\alps$}{alpha\_S}- or \texorpdfstring{$\alpl$}{alpha\_L}-semistable}

We begin by exhibiting a system which is not $\alpha$-semistable for any $\alpha \ge 0$, but which is linearly stable.

\begin{example} \label{NNY}
Let $C$ be a curve of genus $g \ge 4$ which is general in moduli. Then we may choose an integer $\ell$ satisfying
\begin{equation} \label{NNYineq}
\frac{g}{2} + 1 \ < \ \ell \ < \ \frac{2g}{3} + \frac{3}{2} .
\end{equation}
As $d_1 (C) = \left\lceil \frac{g}{2} + 1 \right\rceil$, by (\ref{NNYineq}) the Brill--Noether number $\beta^2_{1, d}$ is positive for $d \ge \ell$. Let $L_1$ and $L_2$ be general elements of $W^1_{\ell + 1}$ and $W^1_\ell$ respectively. By (\ref{NNYineq}) and since $g \ge 4$, one checks that in particular $\ell + 1 \le g + 1$. 
 Thus we may assume that $h^0 (L_1) = h^0 (L_2) = 2$. Furthermore, by Lemma \ref{BasePointFree}, we may assume that $L_1$ and $L_2$ are generated.

We consider now extensions $0 \to L_1 \to E \to L_2 \to 0$ as in Lemma \ref{mfExt}. Condition (\ref{AllSectsLift}) here is
\[
\ell \ > \ 2 \cdot 2 + ( 2-1 ) \cdot (g - 1 - (\ell + 1)) ;
\]
that is, $\ell > \frac{g}{2} + 1$, which follows from (\ref{NNYineq}). By Lemma \ref{mfExt}, there exists a nontrivial extension $E$ as above such that $(E, H^0 (E))$ is generated of type $(2, 2\ell + 1, 4)$.

Now for any $\alpha \ge 0$, we have
\[
\mu_\alpha ( L_1 ) \ = \ (\ell + 1) + \alpha \cdot 2 \ > \ \frac{2 \ell + 1}{2} + \alpha \cdot \frac{4}{2} \ = \ \mu_\alpha (E) .
\]
In particular, the subsystem $(L_1, H^0 (L_1))$ is both $\alps$- and $\alpl$-desemistabilizing. We claim that $(E, H^0 (E))$ is nonetheless linearly stable. Any line subbundle $N \subset E$ with $h^0 (N) \ge 2$ is either $L_1$ or lifts from a subsheaf of $L_2$. But since $L_2$ is generated and the extension is nontrivial, the latter cannot happen and we must have $N = L_1$. But
\[
\lambda (L_1) \ = \ \ell + 1 \ > \ \frac{2 \ell + 1}{2} \ = \ \lambda (E) .
\]
It follows that $L_1$ is not linearly destabilized by a line subbundle.

Suppose now that $F \subset E$ is a generated subsheaf of rank two with $h^0 (F) = 3$. If $L_1 \subset F$ then it is easy to see that $F$ is an extension $0 \to L_1 \to F \to \Oc \to 0$. In this case,
\[
\lambda(F) \ = \ \frac{\ell + 1}{3 - 2} \ = \ell + 1 \ > \ \lambda (E) .
\]
If $L_1 \not\subset F$ then, by the last paragraph, $h^0 (N) \le 1$ for every line subbundle $N \subset F$. Therefore, $\deg F \ge d_2$ by \cite[Lemma 4.4]{LNe}, and by (\ref{NNYineq}) we obtain
\[
\lambda (F) \ \ge \ d_2 (C) \ \ge \ \frac{2g}{3} + 2 \ > \ \ell + \frac{1}{2} \ = \ \lambda(E) .
\]
Thus $F$ does not linearly destabilize $E$. We conclude that $E$ is linearly stable. \qed
\end{example}

\begin{remark} \label{NotAlphaButLin}
Let us prove that with $E$ as in Example \ref{NNY}, in fact $M_E$ is stable (cf.\ \ref{LinStNec}). Since $C$ is assumed to be general in moduli, we have
\[
d_3 (C) \ = \ \left\lceil \frac{3g}{4} + 3 \right\rceil .
\]
Thus, as $\ell < \frac{2g}{3} + \frac{3}{2}$ by (\ref{NNYineq}), in particular $\ell < d_3 - 1$ and $\deg E = 2 \ell + 1 < 2 \cdot d_3 (C)$. As $(E, H^0 (E))$ is linearly stable and by Lemma \ref{mfExt} the bundle $E$ has no trivial quotient, one checks using \cite[Proposition 5.10]{CHL} the DSB $M_E$ is a stable vector bundle.

In particular, a system which is not $\alpha$-semistable for any $\alpha$ may have stable DSB.
\end{remark}

\subsection{Stability of DSB for type \texorpdfstring{$(2, d, 5)$}{(2, d, 5)}}

In order to show linear stability of our last two examples, we shall in fact prove a slightly more general statement on stability of the DSB of generated coherent systems of type $(2, d, 5)$ for low values of $d$. In addition to completing the list of examples in (\ref{Overview}), this will allow us to prove Theorem \ref{MainD}.

\begin{proposition} \label{2d5}
Let $C$ be a curve which is general in the sense of Theorem \ref{MercatRankTwo}; that is, satisfying $\gamma_2' (C) = d_1 (C) - 2$. Let $d$ be an integer satisfying $d < 3 \cdot d_1 (C)$. Suppose that $(E, V)$ is a generated coherent system of type $(2, d, 5)$ which is $\alpl$-stable. Then $\mev$ is a stable vector bundle.
\end{proposition}

\begin{proof}
Let $S \subset \mev$ be a subbundle of maximal slope. Following \cite{CT}, we form the \emph{Butler diagram of $E$ by $S$}:
\[ \xymatrix{
 & 0 \ar[d] & 0 \ar[d] & N \ar[d] & \\
0 \ar[r] & S \ar[r] \ar[d] & \Oc \otimes W \ar[r] \ar[d] & F_S \ar[d] \ar[r] & 0 \\
0 \ar[r] & \mev \ar[r] & \Oc \otimes V \ar[r] & E \ar[r] & 0
} \]
Here $W$ is defined by
\[
W^\vee \ = \ \Image \left( V^\vee \ \to \ H^0 (M_E^\vee) \to H^0 (S^\vee) \right) .
\]
As $M_E^\vee$ is generated by $V^\vee$, also $S^\vee$ is generated by $W^\vee$. Note that $F_S / N \cong E_W$, the subsheaf of $E$ generated by $W$, and that $W \hookrightarrow V$ by the proof of \cite[Lemma 1.9]{But94}.

Firstly, if $S$ is a line bundle then $h^0 ( S^\vee ) \ge \dim W \ge 2$, and so $\deg (S^\vee) \ge d_1 (C)$. As $d < 3 \cdot d_1 (C)$ by hypothesis, $\mu (S) < \mu (\mev)$ and $S$ does not destabilize $\mev$.

If $S$ has rank $2$ then clearly $\dim W \ge 3$. If $\dim W = 3$ then $(E_W , W)$ is a subsystem of type $(1, e, 3)$. But then $(E, V)$ contains a net, which is impossible by Lemma \ref{SpecialImpl} since $(E, V)$ is assumed $\alpl$-stable.

Suppose, then, that $\dim W \ge 4$. As $S$ is assumed to have maximal slope, it is semistable. If $S^\vee$ contributes to $\gamma_2 (C)$ then, by Theorem \ref{MercatRankTwo} and our generality assumptions,
\[
\mu (S^\vee) - h^0 (S^\vee) + 2 \ \ge \ \gamma_2 (C) \ = \ \gamma_1 (C) \ = \ d_1 (C) - 2 ;
\]
whence $\mu (S) \le -d_1 (C) < -\frac{d}{3} = \mu (\mev)$ and $S$ does not destabilize $\mev$. If $S^\vee$ does not contribute to $\gamma_2 (C)$, then since $S^\vee$ is semistable and $h^0 (S^\vee) \ge 4$, we must have
\[
\mu (S^\vee) \ \ge \ g-1 \ > \ \left\lceil \frac{g}{2} + 1 \right\rceil \ \ge \ d_1 (C) > \frac{d}{3} \ = \ \mu (\mevd)
\]
and $S$ does not destabilize $\mev$. We conclude that $\mev$ is a stable vector bundle.
\end{proof}

\begin{remark}
The coherent systems whose existence is shown in Example \ref{NYN} are generated and of type $(2, 2g+1, 5)$ and $\alpl$-stable. As they are not linearly semistable, their DSBs are not semistable. However, there is no contradiction to Proposition \ref{2d5}, as $2g + 1 \ge \frac{3g}{2} + 3 \ge 3 d_1$ for $g \ge 4$.
\end{remark}

\subsection{Systems which are \texorpdfstring{$\alpl$}{alpha\_L}-stable}

Our penultimate example is of a system which is $\alpl$- and linearly stable, but not $\alps$-stable.

\begin{example} \label{NYY}
Let $C$ be a curve of genus $g \ge 25$ which is general in moduli. As previously, we check that there exists an integer $\ell$ satisfying
\begin{equation} \label{NYYineq}
\frac{3g}{4} + 1 \ > \ \ell \ \ge \ d_2 \ = \ \left\lceil \frac{2g}{3} + 2 \right\rceil .
\end{equation}
For such an $\ell$, in particular $\dim W^2_\ell \ge 0$. By Lemma \ref{BasePointFree}, we may choose generated elements $L_1 \in W^1_{\ell + 1}$ and $L_2 \in W^2_\ell$ satisfying $h^0 (L_1) = 2$ and $h^0 (L_2) = 3$. We set
\[
k_1 \ = \ 2 \ \quad \hbox{and} \quad k_2 \ = \ 3 , \quad \ \hbox{and} \quad \ell_1 \ = \ \ell + 1 \quad \hbox{and} \quad \ell_2 \ = \ \ell .
\]
Then the inequality (\ref{AllSectsLift}) becomes
\[
\ell \ > \ 2 \cdot 3 + (3 - 1) ( g - 1 - (\ell + 1) ) ;
\]
that is, $\ell > \frac{2g + 2}{3}$, which follows from (\ref{NYYineq}).
 As before, using Lemma \ref{mfExt} we obtain a nontrivial extension $0 \to L_1 \to E \to L_2 \to 0$ defining a generated coherent system of type $(2, 2 \ell + 1, 5)$.

Firstly, as $\deg L_1 = \ell + 1 > \frac{2\ell + 1}{2} = \mu (E)$, the bundle $E$ is not semistable. Thus $(E, H^0 (E))$ is not $\alps$-semistable by Lemma \ref{StVBImpliesAlps} (b).

Let now $N \subset E$ be any line subbundle. If $h^0 (N) \ge 3$ then $N$ lifts from a subsheaf of $L_2$. As $L_2$ is generated, in fact $N \cong L_2$; but this is impossible since the extension is nontrivial. Therefore $h^0 (N) \le 2$. In particular, $(E, H^0(E))$ does not contain a net, and thus is $\alpl$-stable by Lemma \ref{SpecialImpl}.

 As for linear stability: By (\ref{NYYineq}), we have
\[
\deg E \ = \ 2 \ell + 1 \ < \ 2 \cdot \left( \frac{3g}{4} + 1 \right) + 1 \ = \ \frac{3g}{2} + 3 \ = \\
3 \cdot \left( \frac{g}{2} + 1 \right) \ \le \ 3 \cdot d_1 (C) .
\]
As we have seen that $(E, H^0 (E))$ is $\alpl$-stable, by Proposition \ref{2d5} the DSB $M_E$ is a stable vector bundle, and $(E, H^0 (E))$ is linearly stable by (\ref{LinStNec}). \qed
\end{example}

Our final example is a generated coherent system which is stable in all three senses.

\begin{example}\label{YYY}
Let $C$ be a general curve of genus $g \ge 25$, and let $\ell$ be an integer satisfying
\begin{equation} \label{YYYineq}
\frac{3g}{4} + 1 \ > \ \ell \ \ge \ d_2 \ = \ \left\lceil \frac{2g}{3} + 2 \right\rceil.
\end{equation}
By Lemma \ref{BasePointFree}, we may choose generated $L_1 \in W^1_\ell$ and $L_2 \in W^2_{\ell + 1}$ with $h^0 (L_1) = 2$ and $h^0 (L_2) = 3$. One checks as before that by (\ref{YYYineq}) and Lemma \ref{mfExt} there exists a nontrivial extension $0 \to L_1 \to E \to L_2 \to 0$ defining a generated coherent system of type $(2, 2 \ell + 1, 5)$.
Let $N \subset E$ be a line subbundle. As in the previous example, by nontriviality of the extension $E$ and since $L_2$ is generated, we obtain $\deg N \le \ell$ and $h^0 ( N ) < h^0 (L_2) = 3$. Thus for any $\alpha \ge 0$ we have
\[
\mu_\alpha (N) \ \le \ \ell + 2 \cdot \alpha \ < \ \frac{2 \ell + 1}{2} + \alpha \cdot \frac{5}{2} \ = \ \mu_\alpha (E) .
\]
Therefore, $E$ is $\alpha$-stable for all $\alpha \ge 0$. As in the previous example, we check that $\deg E < 3 d_1 (C)$, and then stability of the DSB $M_E$ and linear stability of $(E, H^0 (E))$ follow from $\alpl$-stability by Proposition \ref{2d5} as before. \qed

\end{example} 

\noindent With Example \ref{YYY} we have completed the list in (\ref{Overview}), and Theorem \ref{MainB} is proven. \qed

\section{A case of Butler's conjecture} \label{OneCaseButler}

We shall now use the results obtained so far to prove Theorem \ref{MainD}.

\begin{proof}[Proof of Theorem \ref{MainD}]
Throughout this proof, we assume that $d \in \left\{ 2 d_2 (C) - 1 , 2 d_2 (C) \right\}$. To ease notation, we abbreviate $d_2 (C)$ to $d_2$.

Firstly, we show that $S_0 (2, d, 5)$ is nonempty. As $C$ is general, we may choose generated line bundles $L_1 \in W_{d_2 - 1}^1$ and $L_2 \in W^2_{d_2 + \varepsilon}$, where $\varepsilon \in \{ 0, 1 \}$. Following the notation of Lemma \ref{mfExt}, we set
\[ \ell_1 = d_2 - 1 , \quad \ell_2 = d_2 + \varepsilon , \quad k_1 = 2 \quad \hbox{and} \quad k_2 = 3 . \]
The inequality (\ref{AllSectsLift}) is then $d_2 + \varepsilon \ > \ 2 \cdot 3 + (3 - 1) \cdot (g - 1 - d_2 + 1)$, which becomes
\begin{equation} \label{2d5Ineq}
d_2 \ > \ \frac{2g - \varepsilon}{3} + 2 .
\end{equation}
As $C$ is general, $d_2 = \left\lceil \frac{2g}{3} + 2 \right\rceil$. If $\varepsilon = 1$, then (\ref{2d5Ineq}) is satisfied for all $g$; and if $\varepsilon = 0$, then (\ref{2d5Ineq}) is satisfied for $g \not\equiv 0 \mod 3$. Thus by Lemma \ref{mfExt}, in each of the situations (a) and (b) of Theorem \ref{MainD} there exists a nontrivial extension $0 \to L_1 \to E \to L_2 \to 0$ such that $(E, H^0 (E))$ is a generated coherent system of type $(2, d, 5)$. From nontriviality of the extension it follows that $E$ is a semistable vector bundle and that $h^0 (N) \le 2$ for any line subbundle $N \subset E$. It follows easily that $(E, H^0 (E))$ is $\alps$-stable; whence $S_0 (2, d, 5)$ is nonempty.

Let $(E, V)$ be \emph{any} $\alps$-stable element of $S_0 (2, d, 5)$ and let $(N, W)$ be a coherent subsystem. Then
\begin{equation} \label{2d5alpsSlope}
\deg N + \alps \cdot \dim W \ < \ \frac{d}{2} + \alps \cdot \frac{5}{2}
\end{equation}
Now by Lemma \ref{StVBImpliesAlps} (c), the vector bundle $E$ is semistable, so $\deg N \le \frac{d}{2}$. If inequality is strict, then $\deg N < d_2$ and $\dim W \le h^0 (N) \le 2$. If $\deg N = d_2$ (which can happen only in case (a)) then by (\ref{2d5alpsSlope}) and since $\alps > 0$ we have $\dim W < \frac{5}{2}$, so again $\dim W \le 2$.

It follows that $(E, V)$ does not contain a net, and so is $\alpl$-stable by Lemma \ref{SpecialImpl}. As $g \ge 18$, we have
\[
2 \cdot d_2 - \varepsilon \ < \ 2 \cdot \left( \frac{2g}{3} + 3 \right) \ = \ \frac{4g}{3} + 6 \ \le \ \frac{3g}{2} + 3 \ \le \ 3 \cdot d_1 .
\]
Thus by Proposition \ref{2d5}, if $(E, V)$ is any $\alps$-stable element of $\cS$ then $\mev$ is a stable vector bundle, and then $(\mevd, V^\vee)$ is $\alpha$-stable for $\alpha$ close to zero by Lemma \ref{StVBImpliesAlps} (a). Thus the map $\cD \colon S_0 (2, d, 5) \dashrightarrow S_0 (3, d, 5)$ given by $(E, V) \mapsto (\mevd, V^\vee)$ is defined on all components of $S_0 (2, d, 5)$.
Since
\[
\left( M_{\mevd, V^\vee}, \left( V^\vee \right)^\vee \right) \ \cong \ (E, V) ,
\]
the dual span construction also gives a birational inverse for $\cD$ on all components intersecting the image of $\cD$. Therefore, to complete the proof we must check that $\cD$ is dominant. It will suffice to show that if $(E_1, V_1)$ is a generic element of any component of $S_0 (3, d, 5)$, then $(\mevod, V_1^\vee)$ is $\alps$-stable.

Let $(E_1, V_1)$ be an $\alps$-stable element of $S_0 (3, d, 5)$. Let $S \subset \mevo$ be a maximal line subbundle. We form the Butler diagram of $( E_1 , V_1 )$ by $S$ as before:
\[ \xymatrix{
 & & & N \ar[d] \\
0 \ar[r] & S \ar[r] \ar[d] & \Oc \otimes W \ar[r] \ar[d] & F_S \ar[d] \\
0 \ar[r] & \mevo \ar[r] & \Oc \otimes V_1 \ar[r] & E_1 ,
} \]
where again $W$ is defined by $W^\vee = \Image \left( V_1^\vee \to H^0 ( \mevod ) \to H^0 ( S^\vee ) \right)$. Firstly, suppose that $\dim W = 2$. Then $N = 0$ and $F_S \cong S^\vee$. As $h^0 ( S^\vee ) \ge 2$, we have
\[
\mu (S^\vee) \ \ge \ d_1 (C) \ \ge \ \frac{g}{2} + 1 .
\]
On the other hand,
\[
\mu (E_1) \ \le \ \frac{2 d_2}{3} \ < \ \frac{2}{3} \cdot \left( \frac{2g}{3} + 3 \right) \ = \ \frac{4g}{9} + 2 \ \le \ \frac{g}{2} + 1 ;
\]
the last inequality since $g \ge 18$. As $E_1$ is a semistable vector bundle by Lemma \ref{StVBImpliesAlps} (c), no such $S^\vee$ and thus no such $S$ can exist.

If $\dim W \ge 3$, then $h^0 ( S^\vee ) \ge 3$ and
\[
\deg S \ \le \ -d_2 \ \le \ - \left( \frac{2d_2 - \varepsilon}{2} \right) \ = \ \mu \left( \mevo \right) .
\]
If inequality is strict, then $(\mevod , V_1^\vee)$ is $\alps$-stable by Lemma \ref{StVBImpliesAlps} (a) (after reducing $\alps$ if necessary) and we are done. If equality obtains, then $d = 2 d_2$ and $\deg S = d_2$. In this case, $\mevod$ is an extension
\[
0 \ \to \ T \ \to \ \mevod \ \to \ S^\vee \ \to \ 0
\]
where $T$ is a line bundle of degree $d_2$ with $\dim (V_1 \cap H^0 (T)) = 2$. Now if this extension were trivial, then we would obtain a direct sum decomposition
\[
(\mevod, V_1^\vee) \ \cong \ \left(T, H^0(T) \cap V_1 \right) \oplus \left( S^\vee, H^0 (S^\vee) \right) .
\]
But then $E_1$ would be a direct sum $M_{S^\vee}^\vee \oplus M_{T, H^0 (T) \cap V_1}^\vee$, which is clearly not semistable; and then $(E_1, V_1)$ would fail to be $\alps$-stable by Lemma \ref{StVBImpliesAlps} (c). Thus $\mevod$ is a nontrivial extension, and in particular $( \mevod, V_1^\vee )$ has no subsystem of type $(1, e, 3)$. Thus if $(N, U)$ is any rank one subsystem, for $\alpha > 0$ we have
\[
\mu_\alpha (N, U) \ \le \ d_2 + 2 \alpha \ < \ d_2 + \alpha \cdot \frac{5}{2} \ = \ \mu_\alpha \left( \mevod, V_1^\vee \right) .
\]
It follows that $\left( \mevod, V_1^\vee \right)$ is $\alps$-stable, as desired.

We conclude that there is no component of $S_0 (3, d, 5)$ where the dual span construction does not produce $\alps$-stable systems. This completes the proof that Butler's conjecture holds nontrivially for $d \in \{ 2 d_2 - 1, 2 d_2 \}$.
\end{proof}

\end{document}